\documentclass[12pt, reqno]{amsart}
\usepackage{amsmath, amsthm, amscd, amsfonts, amssymb, graphicx, color}
\usepackage[bookmarksnumbered, colorlinks, plainpages]{hyperref}
\hypersetup{colorlinks=true,linkcolor=red, anchorcolor=green, citecolor=cyan, urlcolor=red, filecolor=magenta, pdftoolbar=true}

\DeclareMathSymbol{\subsetneqq}{\mathbin}{AMSb}{36}

\newcommand{\R}{\mathbb{R}}

\newcommand{\C}{\mathbb{C}}

\newcommand{\beq}{\begin{eqnarray}}
\newcommand{\eeq}{\end{eqnarray}}
\newcommand{\bq}{\begin{equation}}
\newcommand{\eq}{\end{equation}}
\newcommand{\beqn}{\begin{eqnarray*}}
\newcommand{\eeqn}{\end{eqnarray*}}
\newcommand{\bex}{\begin{exo}}
\newcommand{\eex}{\end{exo}}
\newcommand{\ben}{\begin{enumerate}}
\newcommand{\een}{\end{enumerate}}

\newtheorem{th1}{{\bf Theorem}}[section]
\newtheorem{thm}[th1]{{\bf Theorem}}
\newtheorem{lem}[th1]{{\bf Lemma}}
\newtheorem{prop}[th1]{{\bf Proposition}}
\newtheorem{cor}[th1]{{\bf Corollary}}
\newtheorem{rem}[th1]{\bf Remark}
\newtheorem{rems}[th1]{\bf Remarks}
\newtheorem{defi}[th1]{\bf Definition}
\newtheorem{defs}[th1]{\bf Definitions}

\author[T. Saanouni]{Tarek Saanouni}
\address{Department of Mathematics, College of Sciences and Arts of Uglat Asugour, Qassim University, Buraydah, Kingdom of Saudi Arabia.}
\email{\sl t.saanouni@qu.edu.sa}
\email{\sl Tarek.saanouni@ipeiem.rnu.tn}


\title[IBNLS]{Energy scattering for radial focusing inhomogeneous bi-harmonic Schr\"odinger equations}

\date{\today}
\begin{document}
\begin{abstract}
This note studies the asymptotic behavior of global solutions to the fourth-order Schr\"odinger equation
$$i\dot u+\Delta^2 u+F(x,u)=0 .$$
Indeed, for both cases, local and non-local source term, the scattering is obtained in the focusing mass super-critical and energy sub-critical regimes, with radial setting. This work uses a new approach due to Dodson-Murphy \cite{dm}.
\end{abstract}
\maketitle
\vspace{ 1\baselineskip}
\renewcommand{\theequation}{\thesection.\arabic{equation}}
\section{Introduction}
This manuscript is concerned with the energy scattering theory of the Cauchy problem for the following inhomogeneous focusing Shr\"odinger equation 
\begin{equation}
\left\{
\begin{array}{ll}
i\dot u+\Delta^2  u- |x|^{2b}|u|^{2(q-1)}u=0 ;\\
u(0,.)=u_0,
\label{S2}
\end{array}
\right.
\end{equation}
and the inhomogeneous focusing Choquard equation
\begin{equation}
\left\{
\begin{array}{ll}
i\dot u+\Delta^2  u-(I_\alpha*|\cdot|^b|u|^p)|x|^{b}|u|^{p-2}u=0 ;\\
u(0,.)=u_0.
\label{S}
\end{array}
\right.
\end{equation}

Here and hereafter $u: \R\times\R^N \to \C$, for some natural integer $N$. The unbounded inhomogeneous term is $|\cdot|^b$, for some $b<0$. The source terms satisfy $q>1$ and $p\geq2$. The Riesz-potential is defined on $\R^N$ by 
$$I_\alpha:x\to\frac{\Gamma(\frac{N-\alpha}2)}{\Gamma(\frac\alpha2)\pi^\frac{N}22^\alpha|x|^{N-\alpha}},\quad  0<\alpha<N.$$ 

The homogeneous case corresponding to $b = 0$ was considered first in \cite{Karpman,Karpman 1} to take into account the role of small fourth-order dispersion terms in the propagation of intense laser beams in a bulk medium with a Kerr non-linearity.\\


The equation \eqref{S2} enjoys the scaling invariance
$$u_\lambda=\lambda^\frac{2+b}{q-1}u(\lambda^4.,\lambda .),\quad\lambda>0.$$
The homogeneous Sobolev norm invariant under the previous scaling is $\|\cdot\|_{\dot H^{s_c}}$, where  
$$s_c:=\frac N2-\frac{2+b}{q-1}.$$
Similarly, the equation \eqref{S} satisfies the scaling invariance
$$u_\lambda=\lambda^\frac{4+2b+\alpha}{2(p-1)}u(\lambda^4.,\lambda .),\quad\lambda>0.$$
This gives the critical Sobolev index
$$s'_c:=\frac N2-\frac{4+2b+\alpha}{2(p-1)}.$$
The above problems have the same Sobolev critical index in the limit case $\alpha=0$. In this note, one focus on the mass super-critical and energy sub-critical regimes $0<s_c,s_c'<1$.\\

To the author knowledge, in the literature, there exist few works dealing with the above inhomogeneous fourth-order Schr\"odinger equations. Indeed, for a local source term, in the mass-critical case, the existence of non-global solutions with negative energy was investigated in \cite{cow}. Moreover, the local well-posedness in the energy space of the problem \eqref{S2} was established recently in \cite{gp}. Some well-posedness issues of the inhomogeneous bi-harmonic Choquard problem \eqref{S} were treated by the author in a submitted paper \cite{st2}.\\  

The scattering of the global focusing solutions under the ground state threshold exist for the limiting case $b=0$ in \eqref{S2}, was proved \cite{guo} using the concentration compactness method due to \cite{km}. This result was revisited recently in \cite{vdd}. Note that there exist many works dealing with the biharmonic Schr\"odinger problem with a local source term. On the contrary, there is a few literature which treats the Hartree equation of the fourth order. In fact, some local and global existence results in $H^s$ for the the fourth-order non-linear Schr\"odinger-Hartree equation with variable dispersion coefficients were obtained in \cite{cb}. In addition, in the mass-super-critical and energy-sub-critical regimes, a sharp threshold of global well-psedness and scattering of energy solutions versus finite time blow-up dichotomy was given \cite{st0}. For the stationary case, see \cite{cd}.\\    
 
It is the aim of this note, to investigate the asymptotic behavior of global solutions to both inhomogeneous fourth-order Schr\"odinger and Choquard equations. Indeed, by use of Morawetz estimates and a scattering criterion in the spirit of \cite{tao}, one obtains the scattering in the energy space under the ground state threshold. This work follows the ideas of \cite{Holmer} for the NLS equation.\\ 

 The rest of this paper is organized as follows. The next section contains the main results and some useful estimates. Section three is devoted to prove the scattering of global solutions to the inhomogeneous fourth-order Schr\"odinger equation \eqref{S2}. The last section is consecrated to establish the scattering of global solutions to the inhomogeneous fourth-order Choquard equation \eqref{S}.\\

Here and hereafter, $C$ denotes a constant which may vary from line to another. Denote the Lebesgue space $L^r:=L^r({\R^N})$ with the usual norm $\|\cdot\|_r:=\|\cdot\|_{L^r}$ and $\|\cdot\|:=\|\cdot\|_2$. The inhomogeneous Sobolev space $H^2:=H^2({\R^N})$ is endowed with the norm 
$$ \|\cdot\|_{H^2} := \Big(\|\cdot\|^2 + \|\Delta\cdot\|^2\Big)^\frac12.$$
Let us denote also $C_T(X):=C([0,T],X)$ and $X_{rd}$ the set of radial elements in $X$. Moreover, for an eventual solution to \eqref{S2} or \eqref{S}, $T^*>0$ denotes it's lifespan. Finally, $x^\pm$ are two real numbers near to $x$ satisfying $x^+>x$ and $x^-<x$.

\section{Background and main results}
This section contains the contribution of this paper and some standard estimates needed in the sequel.
\subsection{Preliminary}
Take for $R>0$, $\psi_R:=\psi(\frac\cdot R)$, where $0\leq\psi\leq1$ is a radial smooth function satisfying
$$\psi\in C_0^\infty(\R^N),\quad supp(\psi)\subset \{|x|<1\}, \quad\psi=1\,\,\mbox{on}\,\,\{|x|<\frac12\}.$$
The mass-critical and energy-critical exponents for the Schr\"odinger problem \eqref{S2} are
$$q_*:=1+\frac{4+2b}N\quad\mbox{and}\quad q^*:=\left\{
\begin{array}{ll}
1+\frac{4+2b}{N-4},\quad\mbox{if}\quad N\geq5;\\
\infty,\quad\mbox{if}\quad  1\leq N\leq4.
\end{array}
\right.$$
The mass-critical and energy-critical exponents for the Choquard problem \eqref{S} are
$$p_*:=1+\frac{\alpha+4+2b}N\quad\mbox{and}\quad p^*:=\left\{
\begin{array}{ll}
1+\frac{4+2b+\alpha}{N-4},\quad\mbox{if}\quad N\geq5;\\
\infty,\quad\mbox{if}\quad  1\leq N\leq4.
\end{array}
\right.$$
Here and hereafter, define the real numbers
\begin{gather*}
B:=\frac{Np-N-\alpha-2b}2,\quad A:=2p-B;\\
D:=\frac{Nq-N-2b}2,\quad E:=2q-D.
\end{gather*}
For $u\in H^2$, take the action and the constraint of the Shcr\"odinger problem \eqref{S},
\begin{gather*}
\mathcal S[u]:=M[u]+\mathcal E[u]=\|u\|_{H^2}^2-\frac1q\int_{\R^N}|x|^{2b}|u|^{2q}\,dx;\\
\mathcal K[u]:=\|\Delta u\|^2-\frac D{2q}\int_{\R^N}|x|^{2b}|u|^{2q}\,dx.
\end{gather*}
Similarly, the action and the constraint of the Choquard problem \eqref{S2}, read
\begin{gather*}
S[u]:=M[u]+E[u]=\|u\|_{H^2}^2-\frac1p\int_{\R^N}(I_\alpha*|\cdot|^b|u|^p)|x|^b|u|^p\,dx;\\
K[u]:=\|\Delta u\|^2-\frac B{2p}\int_{\R^N}(I_\alpha*|\cdot|^b|u|^p)|x|^b|u|^p\,dx.
\end{gather*}
\begin{defs}
Let us recall that
\begin{enumerate}
\item[1.]
a ground state of \eqref{S2} is a solution to 
\begin{equation}\label{grnd2}
Q+\Delta^2Q-|x|^{2b}|Q|^{2(q-1)}Q=0,\quad0\neq Q\in H^2,
\end{equation}
which minimizes the problem
$$\inf_{0\neq v\in H^2}\Big\{\mathcal S[v] \quad\mbox{s\,. t}\quad \mathcal K[v]=0\Big\}.$$
\item[2.]
a ground state of \eqref{S} is a solution to 
\begin{equation}\label{grnd}
\phi+\Delta^2\phi-(I_\alpha*|\cdot|^b|\phi|^p)|x|^b|\phi|^{p-2}\phi=0,\quad0\neq\phi\in H^2,
\end{equation}
which minimizes the problem
$$\inf_{0\neq u\in H^2}\Big\{S[u] \quad\mbox{s\,. t}\quad K[u]=0\Big\}.$$
\end{enumerate}
\end{defs}
Let $\phi$ is a ground state and $u$ a solution to \eqref{S2} and $Q$ a ground state and $v$ a solution to \eqref{S}. The following scale invariant quantities describe the dichotomy of global/non-global existence of solutions \cite{guo}.
\begin{gather*}
\mathcal{ME}[u]:=\frac{E[u]^{s_c}M[u]^{2-s_c}}{E[\phi]^{s_c}M[\phi]^{2-s_c}},\quad \mathcal M\mathcal G[u]:=\frac{\|\Delta u\|^{s_c}\|u\|^{2-s_c}}{\|\Delta\phi\|^{s_c}\|\phi\|^{2-s_c}};\\
\mathcal{ME}'[v]:=\frac{\mathcal E[v]^{s_c}M[v]^{2-s_c}}{\mathcal E[Q]^{s_c}M[Q]^{2-s_c}},\quad\mathcal M\mathcal G'[v]:=\frac{\|\Delta v\|^{s_c}\|v\|^{2-s_c}}{\|\Delta Q\|^{s_c}\|Q\|^{2-s_c}}.
\end{gather*}
Let us recall some local well-posedness results \cite{st2,gp} about the above inhomogeneous fourth-order Schr\"odinger problems.
\begin{prop}
Let $N\geq3$, $\max\{-4,-\frac N2\}<2b<0$, $\max\{1,1+\frac{1+2b}N\}<q<q^*$ and $u_0\in H^2$. Then, there exists $T^* = T^*(\|u_0\|_{H^2})$ such that \eqref{S2} admits a unique maximal solution
$$u\in C_{T^*}(H^2).$$ 
Moreover,
\begin{enumerate}
\item[1.]  the solution satisfies the conservation laws
\begin{gather*}
Mass:=M[u(t)] :=\int_{\R^N}|u(t,x)|^2dx = M[u_0];\\
Energy:=\mathcal E[u(t)] :=\|\Delta u(t)\|^2+\frac1q\int_{\R^N}|x|^{2b}|u(t,x)|^{2q}\,dx=\mathcal E[u_0]. 
\end{gather*}
\item[2.] $u\in L^q_{loc}((0,T^*),W^{2,r})$ for any admissible pair $(q,r)$ in the meaning of definition \ref{adm}.
\end{enumerate}
\end{prop}
\begin{prop}
Let $N\geq3$, $0<\alpha<N$ and $\max\{-(N+\alpha),-4(1+\frac\alpha N),N-8-\alpha\}<2b<0$. 
Assume that $N\geq5$ or $3\leq N\leq 4$ and $2\alpha+4b+N>0$. If $u_0\in H^2$ and $2\leq p< p^*$. Then, there exists $T^* = T^*(\|u_0\|_{H^2})$ such that \eqref{S} admits a unique maximal solution
$$ u\in C_{T^*}(H^2).$$
Moreover,
\begin{enumerate}
\item[1.] 
the solutions satisfies the conservation of the mass and the following energy
$$E[u(t)] :=\|\Delta u(t)\|^2+\frac1p\int_{\R^N}|x|^b(I_\alpha *|\cdot|^b|u(t)|^p)|u(t,x)|^p\,dx= E[u_0].$$
\item[2.] $u\in L^q_{loc}((0,T^*),W^{2,r})$ for any admissible pair $(q,r)$.
\end{enumerate}
\end{prop}
\begin{rems}
\begin{enumerate}
\item[1.] 
The regularity condition $p\geq2$, gives the restriction $N-8-\alpha<2b$. This seems to be technical, because it don't appear in the energy;
\item[2.]
denote for simplicity $(N,\alpha,b)$ satisfies $(\mathcal C)$ if, $0<\alpha<N$ and $\max\{-(N+\alpha),-4(1+\frac\alpha N),N-8-\alpha\}<2b<0$ and $[N\geq5$ or $3\leq N\leq 4$ and $2\alpha+4b+N>0]$. 
\end{enumerate}
\end{rems}
The next inhomogeneous Gagliardo-Nirenberg type inequalities \cite{st2} are adapted to the above problems.
\begin{prop}\label{gag2}
Let $N\geq1$.
\begin{enumerate}
\item[1.]
Assume that $\max\{-4,-N\}<2b<0$ and $1< q< q^*$. Then, 
\begin{enumerate}
\item[a.]
there exists $C(N,q,b)>0$, such that for any $v\in H^2$,
$$\int_{\R^N}|x|^{2b}|v|^{2q}\,dx\leq C(N,q,b)\|v\|^E\|\Delta v\|^D;$$
\item[b.]
moreover,
$$C(N,q,b)=\frac{2q}E(\frac ED)^{\frac{D}2}\|Q\|^{-2(q-1)},$$
where $Q$ is a solution to \eqref{grnd2}.
\end{enumerate}
\item[2.]
Let $0<\alpha<N$, $\max\{-(N+\alpha),-4(1+\frac\alpha N)\}<2b<0$ and $1+\frac\alpha N< p< p^*$. Then, 
\begin{enumerate}
\item[a.]
there exists $C(N,p,b,\alpha)>0$, such that for any $u\in H^2$,
$$\int_{\R^N}(I_\alpha*|\cdot|^b|u|^p)|x|^b|u|^p\,dx\leq C(N,p,b,\alpha)\|u\|^A\|\Delta u\|^B.$$
\item[b.]
moreover,
$$C(N,p,b,\alpha)=\frac{2p}{A}(\frac AB)^{\frac{B}2}\|\phi\|^{-2(p-1)},$$
where $\phi$ is a solution to \eqref{grnd}.
\end{enumerate}
\end{enumerate}
\end{prop}
\subsection{Main results}
This subsection contains the contribution of this note.
The first main goal of this manuscript is to prove the following scattering result.
\begin{thm}\label{sctr2}
Take $N\geq5$, $\max\{-4,-\frac N2\}<2b<0$ and $\max\{q_*,x_0\}<q<q^*$ such that $q\geq\frac32$. Let $u\in C(\R,H^2_{rd})$ be a global solution to \eqref{S2}. Then, there exists $u_\pm\in H^2$ such that
$$\lim_{t\to\pm\infty}\|u(t)-e^{it\Delta^2}u_\pm\|_{H^2}=0.$$
\end{thm}
\begin{rems}
Note that
\begin{enumerate}
\item[1.]
the condition $N\geq5$ is required in the proof of the scattering criterion;
\item[2.]
the radial assumption is required in one step of the proof of Morawetz estimate and in the scattering criterion;
\item[3.]
$x_0$ is the positive root of the polynomial function
$$(2X-1)(X-1)-\frac{2(2+b)}{N-4};$$
\item[4.]
the set of $q$ satisfying the above conditions is nonempty because
$$P(q^*)=2\Big(\frac{4+2b}{N-4}\Big)^2>0.$$
\end{enumerate}
\end{rems}
The second main goal of this manuscript is to prove the following scattering result.
\begin{thm}\label{sctr}
Let  $(N,\alpha,b)$ satisfying $(\mathcal C)$ and $\max\{p_*,x_\alpha\}<p<p^*$ such that $p\geq\max\{2,\frac32+\frac\alpha N\}$. Take $u\in C(\R,H^2_{rd})$ be a global solution to \eqref{S}. Then, there exists $u_\pm\in H^2$ such that
$$\lim_{t\to\pm\infty}\|u(t)-e^{it\Delta^2}u_\pm\|_{H^2}=0.$$
\end{thm}
\begin{rems}
Note that
\begin{enumerate}
\item[1.]
$x_\alpha$ is the positive root of the polynomial
$$(X-1)(2X-1)-\frac{4+2b+\alpha}{N-4};$$
\item[2.]
the set of $p$ satisfying the above conditions is nonempty because
$$P_\alpha(p^*)=2\Big(\frac{4+2b+\alpha}{N-4}\Big)^2>0;$$
\item[3.]
it seems that the local and non-local source terms in the above Schr\"odinger problems have a similar asymptotic behavior.
\end{enumerate}
\end{rems}
\subsection{Useful estimates}
Let us gather some classical tools needed in the sequel.
\begin{defi}\label{adm}
Take $N\geq1$ and $s\in[0,2)$. A couple of real numbers $(q,r)$ is said to be $s$-admissible $($admissible for $0$-admissible$)$ if 
$$\frac{2N}{N-2s}\leq r<\frac{2N}{N-4},\quad2\leq q,r\leq\infty\quad\mbox{and}\quad N(\frac12-\frac1r)=\frac4q+s.$$
Denote the set of $s$-admissible pairs by $\Gamma_s$ and $\Gamma:=\Gamma_0$. If $I$ is a time slab, one denotes the Strichartz spaces
$$S^s(I):=\cap_{(q,r)\in\Gamma_s}L^q(I,L^r)\quad\mbox{and}\quad {S^s}'(I):=\cap_{(q,r)\in\Gamma_{-s_c}}L^{q'}(I,L^{r'}).$$
\end{defi}
Recall the Strichartz estimates \cite{bp,guo}.
\begin{prop}\label{prop2}
Let $N \geq 1$, $0\leq s<2$ and $t_0\in I\subset \R$, an interval. Then,
\begin{enumerate}
\item[1.]
$\sup_{(q,r)\in\Gamma}\|u\|_{L^q(I,L^r)}\lesssim\|u(t_0)\|+\inf_{(\tilde q,\tilde r)\in\Gamma}\|i\dot u+\Delta^2 u\|_{L^{\tilde q'}(I,L^{\tilde r'})}$;
\item[2.]
$\sup_{(q,r)\in\Gamma}\|\Delta u\|_{L^q(I,L^r)}\lesssim\|\Delta u(t_0)\|+\|i\dot u+\Delta^2 u\|_{L^2(I,\dot W^{1,\frac{2N}{2+N}})}, \quad\forall N\geq3$;
\item[3.]
$\sup_{(q,r)\in\Gamma_s}\|u\|_{L^q(I,L^r)}\lesssim\|u(t_0)\|_{\dot H^s}+\inf_{(\tilde q,\tilde r)\in\Gamma_{-s}}\|i\dot u+\Delta^2 u\|_{L^{\tilde q'}(I,L^{\tilde r'})}$.
\end{enumerate}
\end{prop}
Let us recall a Hardy-Littlewood-Sobolev inequality \cite{el}.
\begin{lem}\label{Hardy-Littlewwod-Sobolev}
Let $0 <\lambda < N\geq1$ and $1<s,r<\infty$ be such that $\frac1r +\frac1s +\frac\lambda N = 2$. Then,
$$\int_{\R^N\times\R^N} \frac{f(x)g(y)}{|x-y|^\lambda}\,dx\,dy\leq C(N,s,\lambda)\|f\|_{r}\|g\|_{s},\quad\forall f\in L^r,\,\forall g\in L^s.$$
\end{lem}
The next consequence \cite{st}, is adapted to the Choquard problem.
\begin{cor}\label{cor}\label{lhs2}
Let $0 <\lambda < N\geq1$ and $1<s,r,q<\infty$ be such that $\frac1q+\frac1r+\frac1s=1+\frac\alpha N$. Then,
$$\|(I_\alpha*f)g\|_{r'}\leq C(N,s,\alpha)\|f\|_{s}\|g\|_{q},\quad\forall f\in L^s, \,\forall g\in L^q.$$
\end{cor}
Finally, let us give an abstract result.
\begin{lem}\label{abs}
Let $T>0$ and $X\in C([0,T],\R_+)$ such that $$X\leq a+bX^{\theta}\mbox{ on } [0,T],$$
where $a$, $b>0$, $\theta>1$, $a<(1-\frac{1}{\theta})(\theta b)^{\frac{1}{1-\theta}}$ and $X(0)\leq (\theta b)^{\frac{1}{1-\theta}}$. Then
$$X\leq\frac{\theta}{\theta -1}a \mbox{ on } [0,T].$$
\end{lem}
\begin{proof}
The function $f(x):=bx^\theta-x +a$ is decreasing on $[0,(b\theta)^{\frac1{1-\theta}}]$ and increasing on $[(b\theta)^\frac1{1-\theta} ,\infty)$. The assumptions imply that $f((b\theta)^\frac1{1-\theta})< 0$ and $f(\frac\theta{\theta-1}a)\leq0$. As $f(X(t))\geq 0$, $f(0) > 0$ and $X(0)\leq(b\theta)^\frac1{1-\theta}$, we conclude the proof by a continuity argument.
\end{proof}
\section{The Schr\"odinger problem \eqref{S2}}
The goal pf this section is to prove Theorem \ref{sctr2}. Let us collect some estimates needed in the proof of the scattering of global solutions to the focusing Schr\"odinger problem \eqref{S2}.
\subsection{Variational Analysis}
\begin{lem}\label{bnd0}
Take $N\geq3$, $\{-4,-\frac N2\}<2b<0$ and $q_*<q<q^*$. Let $u_0\in H^2$ satisfying 
$$\max\{\mathcal M\mathcal E(u_0),{\mathcal M\mathcal G}(u_0)\}<1.$$
Then, there exists $\delta>0$ such that the solution $u\in C(\R,H^2)$ satisfies
$$\max\{\sup_{t\in\R}\mathcal M\mathcal E(u(t)),\sup_{t\in\R}{\mathcal M\mathcal G}(u(t))\}<1-\delta.$$
\end{lem}
\begin{proof}
Denote $C_{N,q,b}:=C(N,q,b)$ given by Proposition \ref{gag2}. The inequality $\mathcal M\mathcal E(u_0)<1$ gives the existence of $\delta>0$ such that 
\begin{eqnarray*}
1-\delta
&>&\frac{M(u_0)^{\frac{2-s_c}{s_c}}E(u_0)}{M(\phi)^{\frac{2-s_c}{s_c}}E(Q)}\\
&>&\frac{M(u_0)^{\frac{2-s_c}{s_c}}}{M(Q)^{\frac{2-s_c}{s_c}}E(Q)}\Big(\|\Delta u(t)\|^2-\frac1q\int_{\R^N}|x|^{2b}|u|^{2q}\,dx\Big)\\
&>&\frac{M(u_0)^{\frac{2-s_c}{s_c}}}{M(Q)^{\frac{2-s_c}{s_c}}E(Q)}\Big(\|\Delta u(t)\|^2-\frac{C_{N,p,b}}q\|u\|^E\|\Delta u(t)\|^D\Big).
\end{eqnarray*}
Thanks to Pohozaev identities, one has
$$E(Q)=\frac{D-2}D\|\Delta Q\|^2=\frac{D-2}E\|Q\|^2.$$
Thus,
{\small\begin{eqnarray*}
1-\delta
&>&\frac D{D-2}\frac{M(u_0)^{\frac{2-s_c}{s_c}}}{M(Q)^{\frac{2-s_c}{s_c}}\|\Delta Q\|^2}\Big(\|\Delta u(t)\|^2-\frac{C_{N,q,b}}p\|u\|^E\|\Delta u(t)\|^D\Big)\\
&>&\frac D{D-2}\frac{M(u_0)^{\frac{2-s_c}{s_c}}\|\Delta u(t)\|^2}{M(Q)^{\frac{2-s_c}{s_c}}\|\Delta Q\|^2}-\frac D{D-2}\frac{M(u_0)^{\frac{2-s_c}{s_c}}}{M(Q)^{\frac{2-s_c}{s_c}}\|\Delta Q\|^2}\frac{C_{N,q,b}}q\|u\|^E\|\Delta u(t)\|^D\\
&>&\frac D{D-2}\frac{M(u_0)^{\frac{2-s_c}{s_c}}\|\Delta u(t)\|^2}{M(Q)^{\frac{2-s_c}{s_c}}\|\Delta Q\|^2}-\frac D{D-2}\frac{M(u_0)^{\frac{2-s_c}{s_c}}}{M(Q)^{\frac{2-s_c}{s_c}}\|\Delta Q\|^2}\frac2E(\frac ED)^\frac{D}2\|Q\|^{-2(q-1)}\|u\|^E\|\Delta u(t)\|^D\\
&>&\frac D{D-2}\frac{M(u_0)^{\frac{2-s_c}{s_c}}\|\Delta u(t)\|^2}{M(Q)^{\frac{2-s_c}{s_c}}\|\Delta Q\|^2}-\frac D{D-2}\frac2E\frac{M(u_0)^{\frac{2-s_c}{s_c}}}{M(Q)^{\frac{2-s_c}{s_c}}\|\Delta Q\|^2}(\frac{\|Q\|}{\|\Delta Q\|})^D\|Q\|^{-2(q-1)}\|u\|^E\|\Delta u(t)\|^D\\
&>&\frac D{D-2}\frac{M(u_0)^{\frac{2-s_c}{s_c}}\|\Delta u(t)\|^2}{M(Q)^{\frac{2-s_c}{s_c}}\|\Delta Q\|^2}-\frac D{D-2}\frac2E\frac{\|u_0\|^{E+2\frac{1-s_c}{s_c}}}{M(Q)^{\frac{2-s_c}{s_c}}\|\Delta Q\|^2}(\frac{\|Q\|}{\|\Delta Q\|})^D\|Q\|^{-2(q-1)}\|\Delta u(t)\|^D.
\end{eqnarray*}}
Using the equalities $s_c=\frac{D-2}{q-1}$ and $\frac DE=(\frac{\|\Delta Q\|}{\|Q\|})^2$, one has
\begin{eqnarray*}
1-\delta
&>&\frac D{D-2}\frac{M(u_0)^{\frac{2-s_c}{s_c}}\|\Delta u(t)\|^2}{M(Q)^{\frac{2-s_c}{s_c}}\|\Delta Q\|^2}-\frac D{D-2}\frac2E\frac{(\|u_0\|^{\frac{2-s_c}{s_c}}\|\Delta u(t)\|)^D}{M(\phi)^{\frac{2-s_c}{s_c}}\|\Delta Q\|^2}(\frac{\|Q\|}{\|\Delta Q\|})^{D}\|Q\|^{-2(q-1)}\\
&>&\frac D{D-2}\frac{M(u_0)^{\frac{2-s_c}{s_c}}\|\Delta u(t)\|^2}{M(Q)^{\frac{2-s_c}{s_c}}\|\Delta Q\|^2}-\frac2{D-2}\frac{(\|u_0\|^{\frac{2-s_c}{s_c}}\|\Delta u(t)\|)^D}{\| Q\|^{2\frac{2-s_c}{s_c}-D+2q}\|\Delta Q\|^{D}}\\
&>&\frac D{D-2}\Big(\frac{\|u_0\|^{\frac{2-s_c}{s_c}}\|\Delta u(t)\|}{\|Q\|^{\frac{2-s_c}{s_c}}\|\Delta Q\|}\Big)^2-\frac2{D-2}\Big(\frac{\|u_0\|^{\frac{2-s_c}{s_c}}\|\Delta u(t)\|}{\|Q\|^{\frac{2-s_c}{s_c}}\|\Delta Q\|}\Big)^D.
\end{eqnarray*}
Take the real function defined on $[0,1]$ by $f(x):=\frac D{D-2}x^2-\frac2{D-2}x^D$, with first derivative $f'(x)=\frac{2D}{D-2}x(1-x^{D-2})$. Thus, with the table change of $f$ and the continuity of 
$t\to X(t):=\frac{\|u_0\|^{\frac{2-s_c}{s_c}}\|\Delta u(t)\|}{\|\phi\|^{\frac{2-s_c}{s_c}}\|\Delta Q\|}$, it follows that $X(t)<1$ for any $t<T^*$. Thus, $T^*=\infty$ and there exists $\epsilon>0$ near to zero such that $X(t)\in f^{-1}([0,1-\delta])=[0,1-\epsilon]$. This finishes the proof.
\end{proof}
Let us prove a coercivity estimate on centered balls with large radials.
\begin{lem}\label{crcv1}
There exists $R_0:=R_0(\delta,M(u),Q)>0$ such that for any $R>R_0$,
$$\sup_{t\in\R}\|\psi_R u(t)\|^{2-s_c}\|\Delta(\psi_Ru(t))\|^{s_c}<(1-\delta)\|\phi\|^{2-s_c}\|\Delta\phi\|^{s_c}.$$
In particular, there exists $\delta'>0$ such that
$$\|\Delta(\psi_Ru)\|^2-\frac D{2q}\int_{\R^N}|x|^{2b}|\psi_Ru|^{2q}\,dx\geq\delta'\int_{\R^N}|x|^{2b}|\psi_R u|^{2q}\,dx.$$ 
\end{lem}
\begin{proof}
Taking account of Proposition \ref{gag2}, one gets
\begin{eqnarray*}
E(u)
&=&\|\Delta u\|^2-\frac1q\int_{\R^N}|x|^{2b}|u|^{2q}\,dx\\
&\geq&\|\Delta u\|^2\Big(1-\frac{C_{N,q,b}}p\|u\|^E\|\Delta u\|^{D-2}\Big)\\
&\geq&\|\Delta u\|^2\Big(1-\frac{C_{N,q,b}}q[\|u\|^{2-s_c}\|\Delta u\|^{s_c}]^{q-1}\Big).
\end{eqnarray*}
So, with the previous Lemma
\begin{eqnarray*}
E(u)
&\geq&\|\Delta u\|^2\Big(1-(1-\delta)\frac2E(\frac ED)^{\frac D2}\|Q\|^{-2(q-1)}[\|Q\|^{2-s_c}\|\Delta Q\|^{s_c}]^{q-1}\Big)\\
&\geq&\|\Delta u\|^2\Big(1-(1-\delta)\frac2E(\frac ED)^{\frac D2}[\frac{\|\Delta Q\|}{\|Q\|}]^{s_c(q-1)}\Big)\\
&\geq&\|\Delta u\|^2\Big(1-(1-\delta)\frac2E(\frac{\|Q\|}{\|\Delta Q\|})^{D-2}[\frac{\|\Delta Q\|}{\|Q\|}]^{D-2}\Big)\\
&\geq&\|\Delta u\|^2\Big(1-(1-\delta)\frac2E\Big).
\end{eqnarray*}
Thus, using Sobolev injections with the fact that $q<q^*$, one gets
$$\|\Delta u\|^2-\frac D{2q}\int_{\R^N}|x|^{2b}|u|^{2q}\,dx\geq\delta\|\Delta u\|^2\geq\delta'\int_{\R^N}|x|^{2b}|u|^{2q}\,dx.$$ 
This gives the second part of the claimed Lemma provided that the first point is proved. A direct computation \cite{vdd} gives
\begin{equation}\label{vd}
\|\Delta(\psi_R u)\|^2-\|\psi_R\Delta u\|^2\leq C(u_0,\phi)R^{-2}.
\end{equation}
Then, one gets the proof of the first point and so the Lemma.
\end{proof}
\subsection{Morawetz identity}
This subsection is devoted to prove a classical Morawetz estimate satisfied by the energy global solutions to the inhomogeneous Schr\"odinger problem \eqref{S2}. One adopts the convention that repeated indexes are summed. Also, if $f,g$ are two differentiable functions, one defines the momentum brackets by
$$ \{f,g\}_p:=\Re(f\nabla\bar g-g\nabla\bar f).$$
Let us start with an auxiliary result.
\begin{lem}\label{mrw1}
Take $N\geq3$, $\max\{-\frac N2,-4\}<2b<0$, $\max\{1,1+\frac{2+2b}N\}<q<q^*$ and $u\in C_{T}(H^2)$ be a local solution to \eqref{S2}. Let $a:\R^N\to\R$ be a convex smooth function  and the real function defined on $[0,T)$, by
$$M:t\to2\int_{\R^N}\nabla a(x)\Im(\nabla u(t,x)\bar u(t,x))\,dx.$$
Then, the following equalities hold on $[0,T)$,
\begin{eqnarray*}
M'
&=&2\int_{\R^N}\Big(2\partial_{jk}\Delta a\partial_ju\partial_k\bar u-\frac12(\Delta^3a)|u|^2-4\partial_{jk}a\partial_{ik}u\partial_{ij}\bar u\\
&+&\Delta^2a|\nabla u|^2+\partial_ja\{|x|^{2b}|u|^{2(q-1)}u,u\}_p^j\Big)\,dx\\
&=&2\int_{\R^N}\Big(2\partial_{jk}\Delta a\partial_ju\partial_k\bar u-\frac12(\Delta^3a)|u|^2-4\partial_{jk}a\partial_{ik}u\partial_{ij}\bar u\\
&+&\Delta^2a|\nabla u|^2+\frac{q-1}{q}(\Delta a)|x|^{2b}|u|^{2q}-\frac1q\nabla a\nabla(|x|^{2b})|u|^{2q}\Big)\,dx.
\end{eqnarray*}
\end{lem}
\begin{proof}
Denote the source term $\mathcal N:=-|x|^{2b}|u|^{2(q-1)}u$ and compute
\begin{eqnarray*}
\partial_t\Im(\partial_k u\bar u)
&=&\Im(\partial_k\dot u\bar u)+\Im(\partial_k u\bar{\dot u})\\
&=&\Re(i\dot u\partial_k\bar u)-\Re(i\partial_k \dot u\bar{u})\\
&=&\Re(\partial_k\bar u(-\Delta^2 u-\mathcal N))-\Re(\bar u\partial_k(-\Delta^2 u-\mathcal N))\\
&=&\Re(\bar u\partial_k\Delta^2 u-\partial_k\bar u\Delta^2 u)+\Re(\bar u\partial_k\mathcal N-\partial_k\bar u\mathcal N).
\end{eqnarray*}
Thus,
\begin{eqnarray*}
M'
&=&2\int_{\R^N}\partial_ka\Re(\bar u\partial_k\Delta^2 u-\partial_k\bar u\Delta^2 u)\,dx-2\int_{\R^N}\partial_ka\{\mathcal N,u\}_p^k\,dx\\
&=&-2\int_{\R^N}\Delta a\Re(\bar u\Delta^2 u)\,dx-4\int_{\R^N}\Re(\partial_ka\partial_k\bar u\Delta^2 u)\,dx+2\int_{\R^N}\partial_ka\{|x|^{2b}|u|^{2(q-1)}u,u\}_p^k\,dx\\
\end{eqnarray*}
The first equality in the above Lemma follows as in Proposition 3.1 in \cite{mwz}. For the second equality, it is sufficient to use the identity
$$\{|x|^{2b}|u|^{2(q-1)}u,u\}_p=-\frac{q-1}q\nabla(|x|^{2b}|u|^{2q})-\frac1q\nabla(|x|^{2b})|u|^{2q}.$$
\end{proof}
The main result of this subsection is the following.
\begin{lem}\label{bnd1}
Take $N\geq3$, $\max\{-4,-\frac N2\}<2b<0$ and $q_*<q<q^*$. Let $u_0\in H^2_{rd}$ satisfying 
$$\max\{\mathcal M\mathcal E(u_0),{\mathcal M\mathcal G}(u_0)\}<1.$$
Then, for any $T>0$, one has
$$\int_0^T|x|^{2b}|u(t)|^{2q}\,dt\leq CT^{\frac13}.$$
\end{lem}
\begin{proof}
Take a smooth real function such that $0\leq f''\leq1$ and
$$f:r\to\left\{
\begin{array}{ll}
\frac{r^2}2,\,\,\mbox{if}\,\, 0\leq r\leq\frac12;\\
1,\,\,\mbox{if}\,\,  r\geq1.
\end{array}
\right.
$$
Moreover, for $R>0$, let the smooth radial function defined on $\R^N$ by $f_R:=R^2f(\frac{|\cdot|}R)$. One can check that
$$0\leq f_R''\leq1,\quad f'(r)\leq r,\quad N\geq\Delta f_R.$$
Let the real function
$$M_R:t\to2\int_{\R^N}\nabla f_R(x)\Im(\nabla u(t,x)\bar u(t,x))\,dx.$$
By Morawetz estimate in Lemma \ref{mrw1}, one has
\begin{eqnarray*}
M_R'
&=&2\int_{\R^N}\Big(2\partial_{jk}\Delta f_R\partial_ju\partial_k\bar u-\frac12(\Delta^3f_R)|u|^2+\Delta^2f_R|\nabla u|^2-\frac1q\nabla f_R\nabla(|x|^{2b})|u|^{2q}\\
&+&\frac{q-1}{q}(\Delta f_R)|x|^{2b}|u|^{2q}-4\partial_{jk}f_R\partial_{ik}u\partial_{ij}\bar u\Big)\,dx\\
&=&2\int_{\R^N}\Big(2\partial_{jk}\Delta f_R\partial_ju\partial_k\bar u-\frac12(\Delta^3f_R)|u|^2+\Delta^2f_R|\nabla u|^2-\frac1q\nabla f_R\nabla(|x|^{2b})|u|^{2q}\\&+&2\Big(N\frac{q-1}q\int_{\{|x|<\frac R2\}}|x|^{2b}|u|^{2q}\,dx-4\int_{\{|x|<\frac R2\}}|\Delta u|^2\,dx-\frac{2b}q\int_{\{|x|<\frac R2\}}|x|^2b|u|^{2q}\,dx\Big)\\
&+&2\int_{\{\frac R2<|x|<R\}}\Big(\frac{q-1}q\Delta f_R|x|^{2b}|u|^{2q}-4\partial_{jk}f_R\partial_{ik}u\partial_{ij}\bar u\,dx-\frac1q\nabla f_R\nabla(|x|^{2b})|u|^{2q}\Big)\,dx\\
&=&2\int_{\R^N}\Big(2\partial_{jk}\Delta f_R\partial_ju\partial_k\bar u-\frac12(\Delta^3f_R)|u|^2+\Delta^2f_R|\nabla u|^2-\frac1q\nabla f_R\nabla(|x|^{2b})|u|^{2q}\\&+&2\Big(\frac{2D}q\int_{\{|x|<\frac R2\}}|x|^{2b}|u|^{2q}\,dx-4\int_{\{|x|<\frac R2\}}|\Delta u|^2\,dx\Big)\\
&+&2\int_{\{\frac R2<|x|<R\}}\Big(\frac{q-1}q\Delta f_R|x|^{2b}|u|^{2q}-4\partial_{jk}f_R\partial_{ik}u\partial_{ij}\bar u\,dx-\frac1q\nabla f_R\nabla(|x|^{2b})|u|^{2q}\Big)\,dx.
\end{eqnarray*}
Using the estimate $\||\nabla|^kf_R\|_\infty\lesssim R^{2-k}$, one has
\begin{gather*}
|\int_{\R^N}\partial_{jk}\Delta f_R\partial_ju\partial_k\bar u\,dx|\lesssim R^{-2};\\
|\int_{\R^N}(\Delta^3f_R)|u|^2\,dx|\lesssim R^{-4};\\
|\int_{\R^N}\Delta^2f_R|\nabla u|^2\,dx|\lesssim R^{-2}.
\end{gather*}
Moreover, by the radial setting, one writes
$$\int_{\{\frac R2<|x|<R\}}\partial_{jk}f_R\partial_{ik}u\partial_{ij}\bar u\,dx\geq (N-1)\int_{\{\frac R2<|x|<R\}}\frac{f'_R(r)}{r^3}|\partial_ru|^2\,dx=\mathcal O(R^{-2}).$$
Now, by Hardy-Littlewood-Sobolev and Strauss inequalities
\begin{eqnarray*}
|\int_{\{\frac R2<|x|<R\}}\Delta f_R|x|^{2b}|u|^{2q}\,dx|
&\lesssim&\int_{\{\frac R2<|x|<R\}}|x|^{2b}|u|^{2(q-1)}|u|^2\,dx\\
&\lesssim&\int_{\{\frac R2<|x|<R\}}|x|^{2b-(N-1)(q-1)}|u|^2\,dx\\
&\lesssim&R^{2b-(N-1)(q-1)}\|u\|^2\\
&\lesssim&R^{-((N-1)(q-1)-2b)}.
\end{eqnarray*}
With the same way, one has
\begin{eqnarray*}
|\int_{\{\frac R2<|x|<R\}}\nabla f_R\nabla(|x|^{2b})|u|^{2q}\,dx|
&\lesssim&R^{-((N-1)(q-1)-2b)}.
\end{eqnarray*}
Thus, since $\|\nabla u\|^2\lesssim \|\Delta u\|\lesssim 1$ and $q>q^*$, one gets
\begin{eqnarray*}
M_R'
&\leq&4\Big(\frac Dq\int_{\{|x|<\frac R2\}}|x|^{2b}|u|^{2q}\,dx-2\int_{\{|x|<\frac R2\}}|\Delta u|^2\,dx\Big)+\mathcal O(R^{-2}).
\end{eqnarray*}
So, with Lemma \ref{crcv1}, via \eqref{vd}, one gets
\begin{eqnarray*}
\sup_{[0,T]}|M_R|
&\geq&8\int_0^T\Big(\int_{\{|x|<\frac R2\}}|\Delta(\psi_R u)|^2\,dx-\frac{D}{2q}\int_{\{|x|<\frac R2\}}|x|^{2b}|\psi_Ru|^{2q}\,dx\Big)\,dt\\
&+&\mathcal O(R^{-2})T+\mathcal O(R^{-((N-1)(q-1)-2b)})T\\
&\geq&8\delta'\int_0^T\int_{\R^N}|x|^{2b}|\psi_Ru(t)|^{2q}\,dx\,dt+\mathcal O(R^{-2})T\\
&\geq&8\delta'\int_0^T\int_{\{|x|<\frac R2\}}|x|^{2b}|u(t)|^{2q}\,dx\,dt+\mathcal O(R^{-2})T.
\end{eqnarray*}
Thus, with previous computation
\begin{eqnarray*}
\int_0^T\int_{\{|x|<\frac R2\}}|x|^{2b}|u(t)|^{2q}\,dx\,dt
&\leq& C\Big(\sup_{[0,T]}|M|+TR^{-2})\Big)\\
&\leq& C\Big(R+TR^{-2}\Big).
\end{eqnarray*}
Taking $R=T^\frac13>>1$, one gets the requested estimate
$$\int_0^T\int_{\R^N}|x|^{2b}|u(t)|^{2q}\,dx\,dt\leq CT^{\frac13}.$$
For $0<T<<1$, the proof follows with Sobolev injections.
\end{proof}
As a consequence, one has the following energy evacuation.
\begin{lem}\label{evac}
Take $N\geq3$ and $\max\{-4,-\frac N2\}<2b<0$ and $q_*<q<q^*$. Let $u_0\in H^2_{rd}$ satisfying 
$$\max\{\mathcal M\mathcal E(u_0),{\mathcal M\mathcal G}(u_0)\}<1.$$
Then, there exists a sequence of real numbers $t_n\to\infty$ such that 
$$\lim_n\int_{\{|x|<R\}}|u(t_n,x)|^2\,dx=0,\quad\mbox{for all}\quad R>0.$$
\end{lem}
\begin{proof}
Take $t_n\to\infty$. By H\"older estimate
\begin{eqnarray*}
\int_{\{|x|<R\}}|u(t_n,x)|^2\,dx
&\leq& R^{\frac N{q'}}\Big(\int_{\{|x|<R\}}|u(t_n,x)|^{2q}\,dx\Big)^\frac1q\\
&\leq& R^{\frac N{q'}-\frac{2b}q}\Big(\int_{\{|x|<R\}}|x|^{2b}|u(t_n,x)|^{2q}\,dx\Big)^\frac1q\\
\end{eqnarray*}
Indeed, by the previous Lemma 
$$\int_{\R^N}|x|^{2b}|u(t_n)|^{2q}\,dx\to0.$$
\end{proof}
\subsection{Scattering Criterion}
This section is devoted to prove the next result. 
\begin{prop}\label{crt}
Let $N\geq3$, $\max\{-4,-\frac N2\}<2b<0$. Take $\max\{q_*,x_0\}<q<q^*$ such that $q\geq\frac32$. Let $u\in C(\R,H^2_{rd})$ be a global radial solution to \eqref{S2}. Assume that 
$$0<\sup_{t\geq0}\|u(t)\|_{H^2}:=E<\infty.$$
There exist $R,\epsilon>0$ depending on $E,N,q,b$ such that if
$$\liminf_{t\to+\infty}\int_{|x|<R}|u(t,x)|^2\,dx<\epsilon,$$
then, $u$ scatters for positive time. 
\end{prop}
\begin{proof}
By Lemma \ref{bnd0}, $u$ is bounded in $H^2$. Take $\epsilon>0$ near to zero and $R(\epsilon)>>1$ to be fixed later. Define the  Strichartz norm 
$$\|\cdot\|_{S^{s}(I)}:=\sup_{(q,r)\in\Gamma_s}\|\cdot\|_{L^q(I,L^r)}.$$
Let us estimate the non-linear part of the solution with Strichartz norms.
\begin{lem}\label{tch0}
Take $N\geq3$, $\max\{-4,-\frac N2\}<2b<0$. Take $q_*<q<q^*$ and $u\in C(\R,H^2)$ be a global solution to \eqref{S2}. Then, there exists $\theta\in(0,2q-1)$ such that
$$\|u-e^{i.\Delta}u_0\|_{S^{s_c}(I)}\lesssim \|u\|_{L^\infty(I,H^{2})}^\theta\|u\|^{2q-1-\theta}_{L^a(I,L^r)},$$
for a certain $(a,r)\in\Gamma_{s_c}$.
\end{lem}
\begin{proof}
Take the real numbers
\begin{gather*}
a:=\frac{2(2q-\theta)}{2-s_c},\quad d:=\frac{2(2q-\theta)}{2+(2q-1-\theta)s_c};\\
r:=\frac{2N(2q-\theta)}{(N-2s_c)(2q-\theta)-4(2-s_c)}.
\end{gather*}
The condition $\theta=0^+$ gives via some computations that the previous pairs satisfy the admissibility conditions. Indeed, $a>\frac 4{2-s_c}$ and $r<\frac{2N}{N-4}$. Moreover,
$$r=\frac{2N(2q-\theta)}{(N-2s_c)(2q-\theta)-4(2-s_c)}>\frac{2N}{N-2s_c}.$$
Thus,
$$  (a,r)\in \Gamma_{s_c},\quad (d,r)\in \Gamma_{-s_c},\quad
(2q-1-\theta)d'=a.$$
Take two real numbers $r_1,\mu$ satisfying
$$1=\frac2\mu+\frac{\theta}{r_1}+\frac{2q-\theta}r.$$
This gives
\begin{eqnarray*}
\frac{2N}\mu
&=&N-\frac{N\theta}{r_1}-\frac{N(2q-\theta)}r\\
&=&N-\frac{N\theta}{r_1}-\frac{(N-2s_c)(2q-\theta)-4(2-s_c)}2\\
&=&-2b-\frac{N\theta}{r_1}+\theta\frac{2+b}{q-1}.
\end{eqnarray*}
Using  Hardy-Littlewood-Sobolev and H\"older estimates, one has
\begin{eqnarray*}
\|\mathcal N\|_{L^{d'}(I,L^{r'}(|x|<1))}
&\lesssim&\||x|^b\|_{L^\mu(|x|<1)}^2\|u\|_{L^\infty(I,L^{r_1})}^\theta\|\|u\|_r^{2q-1-\theta}\|_{L^{d'}(I)}.
\end{eqnarray*}
 Then, choosing $r_1:=\frac{2N}{N-4}$, one gets because $q<q^*$,
$$\frac{2N}\mu+2b=\frac\theta{2(q-1)}(4+2b-(q-1)(N-4))>0.$$
So, $|x|^b\in L^\mu(|x|<1)$.
Then, by Sobolev injections, one gets
\begin{eqnarray*}
\|\mathcal N\|_{L^{d'}(I,L^{r'}(|x|<1))}
&\lesssim&\||x|^b\|_{L^\mu(|x|<1)}^2\|u\|_{L^\infty(I,L^{r_1})}^\theta\|\|u\|_r^{2q-1-\theta}\|_{L^{d'}(I)}\\
&\lesssim&\|u\|_{L^\infty(I,\dot H^{2})}^\theta\|u\|_{L^a(I,L^r)}^{2q-1-\theta}.
\end{eqnarray*}
Moreover, on the complementary of the unit ball, 
one takes $r_1:=2$, so because $q>q^*$, one has
$$\frac{2N}\mu+2b=\frac\theta{2(q-1)}(4+2b-N(q-1))<0.$$
So, $|x|^b\in L^\mu(|x|>1)$.
The proof is achieved via Strichartz estimates.
\end{proof}
The following result is the key to prove the scattering criterion.
\begin{prop}\label{fn}
Let $N\geq3$, $\max\{-4,-\frac N2\}<2b<0$. Take $\max\{q_*,x_0\}<q<q^*$ such that $q\geq\frac32$. Take $u_0\in H^2_{rd}$ satisfying 
$$\max\{\mathcal M\mathcal E(u_0),{\mathcal M\mathcal G}(u_0)\}<1.$$
Then, for any $\varepsilon>0$, there exist $T,\mu>0$ such that the global solution to \eqref{S2} satisfies
$$\|e^{i(\cdot-T)\Delta^2}u(T)\|_{L^a((T,\infty),L^r)}\lesssim \varepsilon^\mu.$$
\end{prop}
\begin{proof}
One keeps notations of the proof of Lemma \ref{tch0}. Let $0<\beta<<1$ and $T>\varepsilon^{-\beta}>0$. By the integral formula
\begin{eqnarray*}
e^{i(\cdot-T)\Delta^2}u(T)
&=&e^{i\cdot\Delta^2}u_0+i\int_0^Te^{i(\cdot-s)\Delta^2}[|x|^{2b}|u|^{2(q-1)}u]\,ds\\
&=&e^{i\cdot\Delta^2}u_0+i\Big(\int_0^{T-\varepsilon^{-\beta}}+\int_{T-\varepsilon^{-\beta}}^T\Big)e^{i(\cdot-s)\Delta^2}[|x|^{2b}|u|^{2(q-1)}u]\,ds\\
&:=&e^{i\cdot\Delta^2}u_0+F_1+F_2.
\end{eqnarray*}
$\bullet$ The linear term. Since $(a,\frac{Nr}{N+rs_c})\in\Gamma$, by Strichartz estimate and Sobolev injections, one has
\begin{eqnarray*}
\|e^{i\cdot\Delta^2}u_0\|_{L^a((T,\infty),L^r)}
&\lesssim&\||\nabla|^{s_c}e^{i\cdot\Delta^2}u_0\|_{L^a((T,\infty),L^\frac{Nr}{N+rs_c})}\lesssim\|u_0\|_{H^2}.
\end{eqnarray*}
$\bullet$ The term $F_2$. By Lemma \ref{tch0}, one has 
\begin{eqnarray*}
\|F_2\|_{L^a((T,\infty),L^r)}
&\lesssim&\|u\|_{L^\infty((T,\infty),H^{2})}^\theta\|u\|^{2q-1-\theta}_{S^{s_c}(T,\infty)}\\
&\lesssim&\|u\|^{2q-1-\theta}_{S^{s_c}(T,\infty)}.
\end{eqnarray*}
Now, by the assumptions of the scattering criterion, one has
$$\int_{\R^N}\psi_R(x)|u(T,x)|^2\,dx<\epsilon^2.$$
Moreover, a computation with use of \eqref{S2} and the properties of $\psi$ give
$$|\frac d{dt}\int_{\R^N}\psi_R(x)|u(t,x)|^2\,dx|\lesssim R^{-1}.$$
Then, for any $T-\varepsilon^{-\beta}\leq t\leq T$ and $R>\varepsilon^{-2-\beta}$, yields
$$\|\psi_Ru(t)\|\leq\Big( \int_{\R^N}\psi_R(x)|u(T,x)|^2\,dx+C\frac{T-t}R\Big)^\frac12\leq C\varepsilon.$$
Since $2<r<\frac{2N}{N-4}$, there exists $\lambda,\lambda'\in(0,1)$ such that ,  for $R>\varepsilon^{-\frac{2\lambda}{(N-1)\lambda'}}$,
\begin{eqnarray*}
\|u\|_{L^\infty((T,\infty),L^r)}
&\leq&\|\psi_Ru\|_{L^\infty((T,\infty),L^r)}+\|(1-\psi_R)u\|_{L^\infty((T,\infty),L^r)}\\
&\lesssim&\|\psi_Ru\|_{L^\infty((T,\infty),L^2)}^\lambda\|\psi_Ru\|_{L^\infty((T,\infty),L^\frac{2N}{N-4})}^{1-\lambda}\\
&+&\|(1-\psi_R)u\|_{L^\infty((T,\infty),L^\infty)}^{\lambda'}\|(1-\psi_R)u\|_{L^\infty((T,\infty),L^2)}^{1-\lambda'}\\
&\lesssim&\varepsilon^{\lambda}+R^{-\frac{\lambda'(N-1)}{2}}\\
&\lesssim&\varepsilon^\lambda.
\end{eqnarray*}
So,
\begin{eqnarray*}
\|F_2\|_{L^a((T,\infty),L^r)}
&\lesssim&\varepsilon^{-(2q-1-\theta)\frac\beta a}\|u\|_{L^\infty((T-\varepsilon^{-\beta},T),L^{r})}^{2q-1-\theta}\\
&\lesssim&\varepsilon^{-(2q-1-\theta)\frac\beta a}\varepsilon^{\lambda(2q-1-\theta)}\\
&\lesssim&\varepsilon^{(2q-1-\theta)(\lambda-\frac\beta a)}.
\end{eqnarray*}
$\bullet$ The term $F_1$. Take $\frac1r={\lambda''}(\frac1r+\frac{s_c}N)$, $\lambda''\in(0,1)$. By interpolation
\begin{eqnarray*}
\|F_1\|_{L^a((T,\infty),L^r)}
&\lesssim&\|F_1\|_{L^a((T,\infty),L^\frac{N+rs_c}{Nr})}^{\lambda''}\|F_1\|_{L^a((T,\infty),L^\infty)}^{1-\lambda''}\\
&\lesssim&\|e^{i(\cdot-(T-\varepsilon^{-\beta}))\Delta^2}u(T-\varepsilon^{-\beta})-e^{i\cdot\Delta^2}u_0\|_{L^a((T,\infty),L^\frac{N+rs_c}{Nr})}^{\lambda''}\|F_1\|_{L^a((T,\infty),L^\infty)}^{1-\lambda''}\\
&\lesssim&\|F_1\|_{L^a((T,\infty),L^\infty)}^{1-\lambda''}.
\end{eqnarray*}
 Recall the free Schr\"odinger operator decay
$$\|e^{it\Delta^2}\cdot\|_r\leq\frac C{t^{\frac N2(\frac12-\frac1r)}}\|\cdot\|_{r'}, \quad \forall r\geq2.$$
By Proposition \ref{gag2} via the previous estimate, one has
\begin{eqnarray*}
\|F_1\|_{\infty}
&\lesssim&\int_0^{T-\varepsilon^{-\beta}}\frac1{(t-s)^{\frac N4}}\||x|^{2b}|u|^{2q-1}\|_{1}\,ds\\
&\lesssim&\int_0^{T-\varepsilon^{-\beta}}\frac1{(t-s)^{\frac N4}}\|u(s)\|_{H^2}^{2q-1}\,ds\\
&\lesssim&(t-T+\varepsilon^{-\beta})^{1-\frac N4}.
\end{eqnarray*}
In the second line one uses the condition $q\geq\frac32.$ Thus, if $\frac N4>1+\frac1a$, it follows that
\begin{eqnarray*}
\|F_1\|_{L^a((T,\infty),L^r)}
&\lesssim&\|F_1\|_{L^a((T,\infty),L^\infty)}^{1-\lambda''}\\
&\lesssim&\Big(\int_T^\infty(t-T+\varepsilon^{-\beta})^{a[1-\frac N4]}\,dt\Big)^{\frac{1-\lambda''}a}\\
&\lesssim&\varepsilon^{(1-\lambda'')\beta[\frac N4-1-\frac1a]}.
\end{eqnarray*}
The condition $\frac N4>1+\frac1a$ reads
$$N\geq5\quad\mbox{and}\quad (2q-1)(q-1)>\frac{2(2+b)}{N-4}.$$
Take the polynomial function
$$P(X):=(2X-1)(X-1)-\frac{2(2+b)}{N-4}:=2(X-x_0)(X-x_1),\quad x_1>0.$$
So using the inequalities $q\geq\frac32$ and $q>x_1$, one gets
$$\|F_1\|_{L^a((T,\infty),L^r)}\lesssim\varepsilon^c,\quad c>0.$$
One concludes the proof by collecting the previous estimates.
\end{proof}
Now, one proves the scattering criterion. Taking account of Duhamel formula, there exists $\gamma>0$ such that
$$\|e^{i.\Delta^2}u(T)\|_{L^a((0,\infty),L^r)}=\|e^{i(.-T)\Delta^2}u(T)\|_{L^a((0,\infty),L^r)}\lesssim\epsilon^\mu.$$
So, with Lemma \ref{tch0} via the absorption result Lemma \ref{abs}, one gets
$$\|u\|_{L^a((0,\infty),L^r)}\lesssim\epsilon^\mu.$$
With Lemma \ref{tch0}, one gets for $u_+:=e^{-iT\Delta^2}u(T)+i\int_T^\infty e^{-is\Delta^2}\mathcal N\,ds$,
\begin{eqnarray*}
\|u(t)-e^{it\Delta^2}u_+\|_{H^2}
&=&\|\int_t^\infty e^{i(t-s)\Delta^2}\mathcal N\,ds\|_{H^2}\\
&\lesssim&\|u\|_{L^\infty((t,\infty),H^2)}^{\theta}\|u\|_{L^a((t,\infty),L^r)}^{2q-1-\theta}\\
&&\to0.
\end{eqnarray*}
This finishes the proof.
\end{proof}
\subsection{Scattering}
Theorem \ref{sctr2} about the scattering of energy global solutions to the focusing problem \eqref{S2} follows with Proposition \ref{crt} via Lemma \ref{evac}.
\section{The generalized Hartree problem \eqref{S}}
This section is devoted to prove Theorem \ref{sctr}.
\subsection{Variational Analysis}
In this sub-section, one collects some estimates needed in the proof of the scattering of global solutions to the focusing Choquard problem \eqref{S}. 
\begin{lem}\label{bnd}
Take $(N,\alpha,b)$ satisfying $(\mathcal C)$ and $p_*<p<p^*$ such that $p\geq2$. Let $u_0\in H^2$ satisfying 
$$\max\{\mathcal M\mathcal E'(u_0),{\mathcal M\mathcal G'}(u_0)\}<1.$$
Then, there exists $\delta>0$ such that the solution $u\in C(\R,H^2)$ satisfies
$$\max\{\sup_{t\in\R}\mathcal M\mathcal E'(u(t)),\sup_{t\in\R}{\mathcal M\mathcal G}'(u(t))\}<1-\delta.$$
\end{lem}
Let us give a coercivity estimate on centered balls with large radials.
\begin{lem}\label{crcv}
There exists $R_0:=R_0(\delta,M(u),\phi)>0$ such that for any $R>R_0$,
$$\sup_{t\in\R}\|\psi_R u(t)\|^{2-s_c'}\|\Delta(\psi_Ru(t))\|^{s_c'}<(1-\delta)\|\phi\|^{2-s_c'}\|\Delta\phi\|^{s_c'}.$$
In particular, there exists $\delta'>0$ such that
$$\|\Delta(\psi_Ru)\|^2-\frac B{2p}\int_{\R^N}(I_\alpha*|\cdot|^b|\psi_Ru|^p)|x|^b|\psi_Ru|^p\,dx\geq\delta'\|\psi_Ru\|_{\frac{2Np}{N+\alpha+2b}}^2.$$ 
\end{lem}
\begin{rem}
The two previous result follows like Lemmas \ref{bnd0}-\ref{crcv1} via Proposition \ref{gag2}.
\end{rem}
\subsection{Morawetz identity}
This subsection is devoted to prove  a classical Morawetz estimate satisfied by the energy global solutions to the inhomogeneous Choquard problem \eqref{S}. 
\begin{prop}\label{mrwtz}
Take $(N,\alpha,b)$ satisfying $(\mathcal C)$, $p_*< p<p^*$ such that $p\geq2$ and $u\in C_{T}(H^2)$ be a local solution to \eqref{S2}. Let $a:\R^N\to\R$ be a convex smooth function  and the real function defined on $[0,T)$, by
$$M:t\to2\int_{\R^N}\nabla a(x)\Im(\nabla u(t,x)\bar u(t,x))\,dx.$$
Then, the following equality holds on $[0,T)$,
\begin{eqnarray*}
M'
&=&2\int_{\R^N}\Big(2\partial_{jk}\Delta a\partial_ju\partial_k\bar u-\frac12(\Delta^3a)|u|^2-4\partial_{jk}a\partial_{ik}u\partial_{ij}\bar u\\
&+&\Delta^2a|\nabla u|^2+\partial_ja\{(I_\alpha*|\cdot|^b|u|^p)|x|^{b}|u|^{p-2}u,u\}_p^j\Big)\,dx\\
&=&2\int_{\R^N}\Big(2\partial_{jk}\Delta a\partial_ju\partial_k\bar u-\frac12(\Delta^3a)|u|^2-4\partial_{jk}a\partial_{ik}u\partial_{ij}\bar u+\Delta^2a|\nabla u|^2\Big)\\
&+&2\Big((1-\frac2p)\int_{\R^N}\Delta a(I_\alpha*|\cdot|^b|u|^{p})|x|^b|u|^p\,dx-\frac2{p}\int_{\R^N}\partial_ka\partial_k(|x|^b[I_\alpha*|\cdot|^b|u|^p])|u|^{p}\,dx\Big).
\end{eqnarray*}
\end{prop}
\begin{proof}
Denote the source term $\mathcal N:=-(I_\alpha*|\cdot|^b|u|^p)|x|^{b}|u|^{p-2}u.$ By previous computation, one gets 
\begin{eqnarray*}
M'
&=&-2\int_{\R^N}\Delta a\Re(\bar u\Delta^2 u)\,dx-4\int_{\R^N}\Re(\partial_ka\partial_k\bar u\Delta^2 u)\,dx-2\int_{\R^N}\partial_ka\{\mathcal N,u\}_p^k\,dx\\
\end{eqnarray*}
On the other hand
\begin{eqnarray*}
(I)
&:=&\int_{\R^N}\partial_ka\Re(\bar u\partial_k\mathcal N-\partial_k\bar u\mathcal N)\,dx\\
&=&\int_{\R^N}\partial_ka\Re(\partial_k[\bar u\mathcal N]-2\partial_k\bar u\mathcal N)\,dx\\
&=&-\int_{\R^N}\Big(\Delta a\bar u\mathcal N+2\Re(\partial_ka\partial_k\bar u\mathcal N)\Big)\,dx\\
&=&\int_{\R^N}\Big(\Delta a(I_\alpha*|\cdot|^b|u|^{p})|x|^b|u|^p-2\Re(\partial_ka\partial_k\bar u\mathcal N)\Big)\,dx\\
&=&\int_{\R^N}\Delta a(I_\alpha*|\cdot|^b|u|^{p})|x|^b|u|^p\,dx+\frac2{p}\int_{\R^N}\partial_ka\partial_k(|u|^{p})(I_\alpha*|\cdot|^b|u|^p)|x|^b\,dx.
\end{eqnarray*}
Moreover,
{\begin{eqnarray*}
(A)
&:=&\int_{\R^N}|x|^b\partial_ka\partial_k(|u|^{p})(I_\alpha*|\cdot|^b|u|^p)\,dx\\
&=&-\int_{\R^N}div(|x|^b\partial_ka(I_\alpha*|\cdot|^b|u|^p))|u|^{p}\,dx\\
&=&-\int_{\R^N}\Delta a(I_\alpha*|\cdot|^b|u|^p)|x|^b|u|^{p}\,dx-\int_{\R^N}\partial_ka\partial_k(|x|^b[I_\alpha*|\cdot|^b|u|^p])|u|^{p}\,dx.
\end{eqnarray*}}
Then,
\begin{eqnarray*}
(I)
&=&\int_{\R^N}\Delta a(I_\alpha*|\cdot|^b|u|^{p})|x|^b|u|^p\,dx+\frac2{p}(A)\\
&=&\int_{\R^N}\Delta a(I_\alpha*|\cdot|^b|u|^{p})|x|^b|u|^p\,dx\\
&+&\frac2{p}(-\int_{\R^N}\Delta a(I_\alpha*|\cdot|^b|u|^p)|x|^b|u|^{p}\,dx-\int_{\R^N}\partial_ka\partial_k(|x|^b[I_\alpha*|\cdot|^b|u|^p])|u|^{p}\,dx)\\
&=&(1-\frac2p)\int_{\R^N}\Delta a(I_\alpha*|\cdot|^b|u|^{p})|x|^b|u|^p\,dx-\frac2{p}\int_{\R^N}\partial_ka\partial_k(|x|^b[I_\alpha*|\cdot|^b|u|^p])|u|^{p}\,dx.
\end{eqnarray*}
This closes the proof.
\end{proof}
The main result of this subsection reads as follows.
\begin{lem}\label{bnd1'}
Take $(N,\alpha,b)$ satisfying $(\mathcal C)$ and $p_*<p<p^*$ such that $p\geq2$. Let $u_0\in H^2_{rd}$ satisfying 
$$\max\{\mathcal M\mathcal E'(u_0),{\mathcal M\mathcal G}'(u_0)\}<1.$$
Then, for any $T>0$, one has
$$\int_0^T\|u(t)\|_{\frac{2Np}{N+\alpha+2b}}^2\,dt\leq CT^{\frac13}.$$
\end{lem}
\begin{proof}
One keeps the notations $f$, $f_R$ and $M_R$ defined in the proof of Lemma \ref{bnd1}.
By Morawetz estimate in Proposition \ref{mrwtz}, one has
\begin{eqnarray*}
M_R'
&=&2\int_{\R^N}\Big(2\partial_{jk}\Delta f_R\partial_ju\partial_k\bar u-\frac12(\Delta^3f_R)|u|^2-4\partial_{jk}f_R\partial_{ik}u\partial_{ij}\bar u+\Delta^2f_R|\nabla u|^2\Big)\,dx\\
&-&2\Big((-1+\frac2p)\int_{\R^N}\Delta f_R(I_\alpha*|\cdot|^b|u|^p)|x|^b|u|^p\,dx+\frac2{p}\int_{\R^N}\partial_kf_R\partial_k[(I_\alpha*|\cdot|^b|u|^p)|x|^b]|u|^{p}\,dx\Big)\\
&=&2\int_{\R^N}\Big(2\partial_{jk}\Delta f_R\partial_ju\partial_k\bar u-\frac12(\Delta^3f_R)|u|^2-\frac2{p}\partial_kf_R\partial_k[(I_\alpha*|\cdot|^b|u|^p)|x|^b]|u|^{p}\,dx+\Delta^2f_R|\nabla u|^2\Big)\,dx\\
&+&2\Big((1-\frac2p)N\int_{\{|x|<\frac R2\}}(I_\alpha*|\cdot|^b|u|^p)|x|^b|u|^p\,dx-4\int_{\{|x|<\frac R2\}}|\Delta u|^2\,dx\Big)\\
&+&2\Big((1-\frac2p)\int_{\{\frac R2<|x|<R\}}\Delta f_R(I_\alpha*|\cdot|^b|u|^p)|x|^b|u|^p\,dx-4\int_{\{\frac R2<|x|<R\}}\partial_{jk}f_R\partial_{ik}u\partial_{ij}\bar u\,dx\Big).
\end{eqnarray*}
Using the estimate $\||\nabla|^kf_R\|_\infty\lesssim R^{2-k}$, one has
\begin{gather*}
|\int_{\R^N}\partial_{jk}\Delta f_R\partial_ju\partial_k\bar u\,dx|\lesssim R^{-2};\\
|\int_{\R^N}(\Delta^3f_R)|u|^2\,dx|\lesssim R^{-4};\\
|\int_{\R^N}\Delta^2f_R|\nabla u|^2\,dx|\lesssim R^{-2}.
\end{gather*}
Moreover, by the radial setting, one writes
$$\int_{\{\frac R2<|x|<R\}}\partial_{jk}f_R\partial_{ik}u\partial_{ij}\bar u\,dx\geq (N-1)\int_{\{\frac R2<|x|<R\}}\frac{f'_R(r)}{r^3}|\partial_ru|^2\,dx=\mathcal O(R^{-2}).$$
Now, by Hardy-Littlewood-Sobolev and Strauss inequalities
\begin{eqnarray*}
|\int_{\{\frac R2<|x|<R\}}\Delta f_R(I_\alpha*|\cdot|^b|u|^p)|x|^b|u|^p\,dx|
&\lesssim&\|u\|_{L^{\frac{2Np}{\alpha+N+2b}}(|x|>\frac R2)}^{2p}\\
&\lesssim&\Big(\int_{|x|>\frac R2}|u(x)|^{\frac{2Np}{\alpha+N+2b}-2}|u(x)|^2\,dx\Big)^{\frac{\alpha+N+2b}N}\\
&\lesssim&\|u\|_{L^\infty(|x|>\frac R2)}^\frac{4B}N\Big(\int_{\R^N}|u(x)|^2\,dx\Big)^{\frac{\alpha+N+2b}N}\\
&\lesssim&R^{-\frac{2B(N-1)}N}.
\end{eqnarray*}
Thus, since $\|\nabla u\|^2\lesssim \|\Delta u\|\lesssim 1$, one gets
\begin{eqnarray*}
M_R'
&\leq&2\Big((1-\frac2p)N\int_{\{|x|<\frac R2\}}(I_\alpha*|\cdot|^b|u|^p)|x|^b|u|^p\,dx-4\int_{\{|x|<\frac R2\}}|\Delta u|^2\,dx\Big)\\
&-&\frac4{p}\int_{\R^N}\partial_kf_R\partial_k[(I_\alpha*|\cdot|^b|u|^p)|x|^b]|u|^{p}\,dx+\mathcal O(R^{-2}).
\end{eqnarray*}
Now, let us define the sets
\begin{gather*}
\Omega:=\{(x,y)\in\R^N\times\R^N,\,\,\mbox{s\,.t}\,\,\frac R2<|x|<R\}\cup \{(x,y)\in\R^N\times\R^N,\,\,\mbox{s\,.t}\,\,\frac R2<|y|<R\};\\
\Omega':=\{(x,y)\in\R^N\times\R^N,\,\,\mbox{s\,.t}\,\,|x|>R,|y|<\frac R2\}\cup \{(x,y)\in\R^N\times\R^N,\,\,\mbox{s\,.t}\,\,|x|<\frac R2, |y|> R\}.
\end{gather*}
Consider the term
\begin{eqnarray*}
(I)
&:=&\int_{\R^N}\nabla f_R(\frac.{|\cdot|^2}I_\alpha*|\cdot|^b|u|^p)|x|^b|u|^{p}\,dx\\
&=&\frac12\int_{\R^N}\int_{\R^N}(\nabla f_R(x)-\nabla f_R(y))(x-y)\frac{I_\alpha(x-y)}{|x-y|^2}|y|^b|u(y)|^p|x|^b|u(x)|^{p}\,dx\,dy\\
&=&\Big(\int_{\Omega}+\int_{\Omega'}+\int_{|x|,|y|<\frac R2}+\int_{|x|,|y|>R}\Big)\Big(\nabla f_R(x)(x-y)\frac{I_\alpha(x-y)}{|x-y|^2}|y|^b|u(y)|^p|x|^b|u(x)|^{p}\,dx\,dy\Big).
\end{eqnarray*}
Compute
\begin{eqnarray*}
(a)
&:=&\int_{\Omega'}\Big(\nabla f_R(x)(x-y)\frac{I_\alpha(x-y)}{|x-y|^2}|y|^b|u(y)|^p|x|^b|u(x)|^{p}\Big)\,dx\,dy\\
&=&\int_{\{|x|>R,|y|<\frac R2\}}\Big(\nabla f_R(x)(x-y)\frac{I_\alpha(x-y)}{|x-y|^2}|y|^b|u(y)|^p|x|^b|u(x)|^{p}\Big)\,dx\,dy\\
&+&\int_{\{|y|>R,|x|<\frac R2\}}\Big(\nabla f_R(x)(x-y)\frac{I_\alpha(x-y)}{|x-y|^2}|y|^b|u(y)|^p|x|^b|u(x)|^{p}\Big)\,dx\,dy\\
&=&\int_{\{|x|>R,|y|<\frac R2\}}\Big((\nabla f_R(x)-\nabla f_R(y))(x-y)\frac{I_\alpha(x-y)}{|x-y|^2}|y|^b|u(y)|^p|x|^b|u(x)|^{p}\Big)\,dx\,dy\\
&=&\int_{\{|x|>R,|y|<\frac R2\}}\Big(y(y-x)\frac{I_\alpha(x-y)}{|x-y|^2}|y|^b|u(y)|^p|x|^b|u(x)|^{p}\Big)\,dx\,dy.
\end{eqnarray*}
Moreover,
\begin{eqnarray*}
(b)
&:=&\frac12\int_{\{|x|<\frac R2,|y|<\frac R2\}}\Big((\nabla f_R(x)-\nabla f_R(y))(x-y)\frac{I_\alpha(x-y)}{|x-y|^2}|y|^b|u(y)|^p|x|^b|u(x)|^{p}\Big)\,dx\,dy\\
&=&\frac12\int_{\{|x|<\frac R2,|y|<\frac R2\}}\Big((x-y)(x-y)\frac{I_\alpha(x-y)}{|x-y|^2}|y|^b|u(y)|^p|x|^b|u(x)|^{p}\Big)\,dx\,dy\\
&=&\frac12\int_{\{|x|<\frac R2,|y|<\frac R2\}}\Big(I_\alpha(x-y)|y|^b|u(y)|^p|x|^b|u(x)|^{p}\Big)\,dx\,dy\\
&=&\frac12\int_{\R^N}(I_\alpha*|\cdot|^b|\psi_Ru|^p)|x|^b|\psi_Ru|^{p}\,dx.
\end{eqnarray*}
Furthermore,
\begin{eqnarray*}
(c)
&:=&\int_{\{\frac R2<|x|<R\}}\int_{\R^N}\Big(\nabla f_R(x)(x-y)\frac{I_\alpha(x-y)}{|x-y|^2}|y|^b|u(y)|^p|x|^b|u(x)|^{p}\Big)\,dx\,dy\\
&=&\int_{\{\frac R2<|x|<R,|y-x|>\frac R4\}}\Big(\nabla f_R(x)(x-y)\frac{I_\alpha(x-y)}{|x-y|^2}|y|^b|u(y)|^p|x|^b|u(x)|^{p}\Big)\,dx\,dy\\
&+&\int_{\{\frac R2<|x|<R,|y-x|<\frac R4\}}\Big(\nabla f_R(x)(x-y)\frac{I_\alpha(x-y)}{|x-y|^2}|y|^b|u(y)|^p|x|^b|u(x)|^{p}\Big)\,dx\,dy\\
&=&\mathcal O\Big(\int_{\{|x|>\frac R2\}}(I_\alpha*|\cdot|^b|u|^p)|x|^b|u|^{p}\,dx\Big).
\end{eqnarray*}
Moreover, since for large $R>0$ on $\{|x|>R,|y|<\frac R2\}$, $|x-y|\simeq |x|>R>>\frac R2>|y|$, one has
\begin{eqnarray*}
(a)
&=&\int_{\{|x|>R,|y|<\frac R2\}}\Big(y(y-x)\frac{I_\alpha(x-y)}{|x-y|^2}|y|^b|u(y)|^p|x|^b|u(x)|^{p}\Big)\,dx\,dy\\
&\lesssim&\int_{\{|x|>R,|y|<\frac R2\}}\Big(|y||x-y|\frac{I_\alpha(x-y)}{|x-y|^2}|y|^b|u(y)|^p|x|^b|u(x)|^{p}\Big)\,dx\,dy\\
&\lesssim&\int_{\{|x|>R,|y|<\frac R2\}}\Big({I_\alpha(x-y)}|y|^b|u(y)|^p|x|^b|u(x)|^{p}\Big)\,dx\,dy\\
&\lesssim&\int_{\R^N}\int_{\R^N}\Big({I_\alpha(x-y)}\chi_{|x|>R}|y|^b|u(y)|^p|x|^b|u(x)|^{p}\Big)\,dx\,dy\\
&\lesssim&\int_{\R^N}\int_{\R^N}\Big({I_\alpha(x-y)}\chi_{|x|>R}|y|^b|u(y)|^p|x|^b|u(x)|^{p}\Big)\,dx\,dy.
\end{eqnarray*}
Taking account of Hardy-Littlewood-Sobolev inequality, H\"older and Strauss estimates and Sobolev injections via the fact that $p_*<p<p^*$, write
\begin{eqnarray*}
(a)
&\lesssim&\int_{\R^N}\int_{\R^N}\Big({I_\alpha(x-y)}\chi_{|x|>R}|y|^b|u(y)|^p|x|^b|u(x)|^{p}\Big)\,dx\,dy\\
&\lesssim&\||u\|_{L^{\frac{2Np}{N+\alpha+2b}}(|x|>R)}^p\|u\|_{L^{\frac{2Np}{N+\alpha+2b}}}^p\\
&\lesssim&\Big(\int_{\{|x|>R\}}|u|^{\frac{2Np}{N+\alpha+2b}}\,dx\Big)^{\frac{N+\alpha+2b}{2N}}\\
&\lesssim&\Big(\int_{\{|x|>R\}}|u|^2(|x|^{-\frac{N-1}2}\|u\|^\frac12\|\nabla u\|^\frac12)^{-2+\frac{2Np}{N+\alpha+2b}}\,dx\Big)^{\frac{N+\alpha+2b}{2N}}\\
&\lesssim&\|u\|^{\frac{N+\alpha+2b}{N}}\frac1{R^{\frac{B(N-1)}{2N}}}(\|u\|\|\nabla u\|)^{\frac B{2N}}\\
&\lesssim&R^{-\frac{B(N-1)}{2N}}.
\end{eqnarray*}
Take the quantity
\begin{eqnarray*}
(II)
&:=&\int_{\R^N}\nabla f_R(I_\alpha*|\cdot|^b|u|^p)x|x|^{b-2}|u|^{p}\,dx\\
&=&\int_{|x|<\frac R2}(I_\alpha*|\cdot|^b|u|^p)|x|^{b}|u|^{p}\,dx+\int_{\frac R2<|x|<R}\nabla f_R(I_\alpha*|\cdot|^b|u|^p)x|x|^{b-2}|u|^{p}\,dx.
\end{eqnarray*}
Now, using the properties of $f_R$, one has
\begin{eqnarray*}
\int_{\frac R2<|x|<R}\nabla f_R(I_\alpha*|\cdot|^b|u|^p)x|x|^{b-2}|u|^{p}\,dx
&\lesssim&\int_{\frac R2<|x|<R}\frac{|\nabla f_R(x)|}{|x|}(I_\alpha*|\cdot|^b|u|^p)|x|^{b}|u|^{p}\,dx\\
&\lesssim&\int_{\frac R2<|x|<R}(I_\alpha*|\cdot|^b|u|^p)|x|^{b}|u|^{p}\,dx\\
&\lesssim&R^{-\frac{2B(N-1)}N}.
\end{eqnarray*}
So,
\begin{eqnarray*}
(II)
&=&\int_{|x|<\frac R2}(I_\alpha*|\cdot|^b|u|^p)|x|^{b}|u|^{p}\,dx+\mathcal O(R^{-2}).
\end{eqnarray*}
Then,
\begin{eqnarray*}
M'
&\leq&2\Big((1-\frac2p)N\int_{\{|x|<\frac R2\}}(I_\alpha*|\cdot|^b|u|^p)|u|^p\,dx-4\int_{\{|x|<\frac R2\}}|\Delta u|^2\,dx\Big)\\
&-&\frac4{p}\Big((\alpha-N)(I)+b(II)\Big)\\
&\leq&2\Big((1-\frac2p)N\int_{\{|x|<\frac R2\}}(I_\alpha*|\cdot|^b|\psi_Ru|^p)|\psi_Ru|^p\,dx-4\int_{\{|x|<\frac R2\}}|\Delta u|^2\,dx\Big)\\
&-&\frac4{p}(b-\frac{N-\alpha}2)\int_{\R^N}(I_\alpha*|\cdot|^b|\psi_Ru|^p)|x|^b|\psi_Ru|^{p}\,dx+\mathcal O(R^{-2})\\
&\leq&\frac{4B}p\int_{\{|x|<\frac R2\}}(I_\alpha*|\cdot|^b|\psi_Ru|^p)|x|^b|\psi_Ru|^p\,dx-8\int_{\{|x|<\frac R2\}}|\Delta u|^2\,dx+\mathcal O(R^{-2}).
\end{eqnarray*}
So, with Lemma \ref{crcv}, via \eqref{vd}, one gets
\begin{eqnarray*}
\sup_{[0,T]}|M|
&\geq&8\int_0^T\Big(\int_{\{|x|<\frac R2\}}|\Delta(\psi_R u)|^2\,dx-\frac{B}{2p}\int_{\{|x|<\frac R2\}}(I_\alpha*|\psi_Ru|^p)|\psi_Ru|^p\,dx\Big)\,dt+\mathcal O(R^{-2})T\\
&\geq&8\delta'\int_0^T\|\psi_Ru(t)\|_{\frac{2Np}{N+\alpha+2b}}^2\,dt+\mathcal O(R^{-2})T\\
&\geq&8\delta'\int_0^T\|u(t)\|_{\frac{2Np}{N+\alpha+2b}(|x|<\frac R2)}^2\,dt+\mathcal O(R^{-2})T.
\end{eqnarray*}
Thus, with previous computation
\begin{eqnarray*}
\int_0^T\|u(t)\|_{\frac{2Np}{N+\alpha+2b}}^2\,dt
&\leq& C\Big(\sup_{[0,T]}|M|+T(R^{-2}+R^{-\frac{B(N-1)}{2N}})\Big)\\
&\leq& C\Big(R+TR^{-2}\Big).
\end{eqnarray*}
Taking $R=T^\frac13>>1$, one gets the requested estimate
$$\int_0^T\|u(t)\|_{\frac{2Np}{N+\alpha+2b}}^2\,dt\leq CT^{\frac13}.$$
For $0<T<<1$, the proof follows with Sobolev injections.
\end{proof}
As a consequence, one has the following energy evacuation.
\begin{lem}\label{evac1}
Take $(N,\alpha,b)$ satisfying $(\mathcal C)$. Let $p_*<p<p^*$ such that $p\geq2$ and $u_0\in H^2_{rd}$ satisfying 
$$\max\{\mathcal M\mathcal E'(u_0),{\mathcal M\mathcal G}'(u_0)\}<1.$$
Then, there exists a sequence of real numbers $t_n\to\infty$ such that the global solution to \eqref{S} satisfies
$$\lim_n\int_{\{|x|<R\}}|u(t_n,x)|^2\,dx=0,\quad\mbox{for all}\quad R>0.$$
\end{lem}
\begin{proof}
Take $t_n\to\infty$. By H\"older estimate
$$\int_{\{|x|<R\}}|u(t_n,x)|^2\,dx\leq R^{\frac{2B}p}\|u(t_n)\|_{\frac{2Np}{N+\alpha+2b}}^2\to0.$$
Indeed, by the previous Lemma 
$$\|u(t_n)\|_{\frac{2Np}{N+\alpha+2b}}\to0.$$
\end{proof}
\subsection{Scattering Criterion}
This subsection is devoted to prove the next result. 
\begin{prop}\label{crt}
Take $(N,\alpha,b)$ satisfying $(\mathcal C)$. Let $\max\{p_*,x_\alpha\}<p<p^*$ such that $p\geq\max\{2,\frac32+\frac\alpha N\}$. Let $u\in C(\R,H^2_{rd})$ be a global radial solution to \eqref{S}. Assume that 
$$0<\sup_{t\geq0}\|u(t)\|_{H^2}:=E<\infty.$$
There exist $R,\epsilon>0$ depending on $E,N,p,b,\alpha$ such that if
$$\liminf_{t\to+\infty}\int_{|x|<R}|u(t,x)|^2\,dx<\epsilon,$$
then, $u$ scatters for positive time. 
\end{prop}
\begin{proof}
By Lemma \ref{bnd}, $u$ is bounded in $H^2$. Take $\epsilon>0$ near to zero and $R(\epsilon)>>1$ to be fixed later. Let us give a technical result.
\begin{lem}\label{tch}
Let $(N,b,\alpha)$ satisfying $(\mathcal C)$ and $p_*<p<p^*$ satisfying $p\geq2$. Then, there exists $\theta\in(0,2p-1)$ such that the global solution to \eqref{S} satisfies
$$\|u-e^{i.\Delta}u_0\|_{S^{s_c}(I)}\lesssim \|u\|_{L^\infty(I,H^{2})}^\theta\|u\|^{2p-1-\theta}_{L^a(I,L^r)},$$
for certain $(a,r)\in\Gamma_{s'_c}$.
\end{lem}
\begin{proof}
Take the real numbers
\begin{gather*}
a:=\frac{2(2p-\theta)}{2-s_c'},\quad d:=\frac{2(2p-\theta)}{2+(2p-1-\theta)s_c'};\\
r:=\frac{2N(2p-\theta)}{(N-2s_c')(2p-\theta)-4(2-s_c')}.
\end{gather*}
The condition $\theta=0^+$ gives via previous computation
$$ (a,r)\in \Gamma_{s_c'},\quad (d,r)\in \Gamma_{-s_c'}\quad\mbox{and}\quad (2p-1-\theta)d'=a.$$
Take two real numbers $r_1,\mu$ satisfying
$$1+\frac\alpha N=\frac2\mu+\frac{\theta}{r_1}+\frac{2p-\theta}r.$$
This gives
\begin{eqnarray*}
\frac{2N}\mu
&=&\alpha+N-\frac{N\theta}{r_1}-\frac{N(2p-\theta)}r\\
&=&\alpha+N-\frac{N\theta}{r_1}-\frac{(N-2s_c')(2p-\theta)-4(2-s_c')}2\\
&=&-2b-\frac{N\theta}{r_1}+\theta\frac{4+2b+\alpha}{2(p-1)}.
\end{eqnarray*}
Using  Hardy-Littlewood-Sobolev and H\"older estimates, one has
\begin{eqnarray*}
\|\mathcal N\|_{L^{b'}(I,L^{r'}(|x|<1))}
&\lesssim&\||x|^b\|_{L^\mu(|x|<1)}^2\|u\|_{L^\infty(I,L^{r_1})}^\theta\|\|u\|_r^{2p-1-\theta}\|_{L^{d'}(I)}.
\end{eqnarray*}
 Then, choosing $r_1:=\frac{2N}{N-4}$, one gets because $p<p^*$,
$$\frac{2N}\mu+2b=\frac\theta{2(p-1)}(4+2b+\alpha-(p-1)(N-4))>0.$$
So, $|x|^b\in L^\mu(|x|<1)$.
 Then, by Sobolev injections, one gets
\begin{eqnarray*}
\|\mathcal N\|_{L^{b'}(I,L^{r'}(|x|<1))}
&\lesssim&\||x|^b\|_{L^\mu(|x|<1)}^2\|u\|_{L^\infty(I,L^{r_1})}^\theta\|\|u\|_r^{2p-1-\theta}\|_{L^{d'}(I)}\\
&\lesssim&\|u\|_{L^\infty(I,\dot H^{2})}^\theta\|u\|_{L^a(I,L^r)}^{2p-1-\theta}.
\end{eqnarray*}
 Moreover, on the complementary of the unit ball, 
one takes $r_1:=2$, so because $p>p_*$, one has
$$\frac{2N}\mu+2b=\frac\theta{2(p-1)}(4+2b+\alpha-N(p-1))<0.$$
So, $|x|^b\in L^\mu(|x|>1)$.
The proof is achieved via Strichartz estimates.
\end{proof}
The following result is the key to prove the scattering criterion.
\begin{prop}\label{fn1}
Take $(N,\alpha,b)$ satisfying $(\mathcal C)$. Let $\max\{p_*,x_\alpha\}<p<p^*$ such that $p\geq\max\{2,\frac32+\frac\alpha N\}$. Let $u_0\in H^2_{rd}$ satisfying 
$$\max\{\mathcal M\mathcal E'(u_0),{\mathcal M\mathcal G}'(u_0)\}<1.$$
Then, for any $\varepsilon>0$, there exist $T,\mu>0$ such that the global solution to \eqref{S} satisfies
$$\|e^{i(\cdot-T)\Delta^2}u(T)\|_{L^a((T,\infty),L^r)}\lesssim \varepsilon^\mu.$$
\end{prop}
\begin{proof}
One keeps notations of the proof of Lemma \ref{tch}. Let $0<\beta<<1$ and $T>\varepsilon^{-\beta}>0$. By the integral formula
\begin{eqnarray*}
e^{i(\cdot-T)\Delta^2}u(T)
&=&e^{i\cdot\Delta^2}u_0+i\int_0^Te^{i(\cdot-s)\Delta^2}[(I_\alpha*|\cdot|^b|u|^p)|x|^b|u|^{p-2}u]\,ds\\
&=&e^{i\cdot\Delta^2}u_0+i\Big(\int_0^{T-\varepsilon^{-\beta}}+\int_{T-\varepsilon^{-\beta}}^T\Big)e^{i(\cdot-s)\Delta^2}[(I_\alpha*|\cdot|^b|u|^p)|x|^b|u|^{p-2}u]\,ds\\
&:=&e^{i\cdot\Delta^2}u_0+F_1+F_2.
\end{eqnarray*}
$\bullet$ The linear term. Since $(a,\frac{Nr}{N+rs_c'})\in\Gamma$, by Strichartz estimate and Sobolev injections, one has
\begin{eqnarray*}
\|e^{i\cdot\Delta^2}u_0\|_{L^a((T,\infty),L^r)}
&\lesssim&\||\nabla|^{s_c'}e^{i\cdot\Delta^2}u_0\|_{L^a((T,\infty),L^{\frac{Nr}{N+rs_c'}})}\lesssim\|u_0\|_{H^2}.
\end{eqnarray*}
$\bullet$ The term $F_2$. By Lemma \ref{tch}, one has 
\begin{eqnarray*}
\|F_2\|_{L^a((T,\infty),L^r)}
&\lesssim&\|u\|_{L^\infty((T,\infty),H^{2})}^\theta\|u\|^{2p-1-\theta}_{L^a((T,\infty),L^r)}\\
&\lesssim&\|u\|^{2p-1-\theta}_{L^a((T,\infty),L^r)}.
\end{eqnarray*}
Following lines in the proof of Proposition \ref{fn}, there exists $\lambda,\lambda'\in(0,1)$ such that ,  for $R>\varepsilon^{-\frac{2\lambda}{(N-1)\lambda'}}$,
\begin{eqnarray*}
\|F_2\|_{L^a((T,\infty),L^r)}
&\lesssim&\varepsilon^{(2p-1-\theta)(\lambda-\frac{\beta}a)}.
\end{eqnarray*}
$\bullet$ The term $F_1$. Take $\frac1r={\lambda''}(\frac1r+\frac{s_c'}N)$ and $\lambda''\in(0,1)$. By interpolation
\begin{eqnarray*}
\|F_1\|_{L^a((T,\infty),L^r)}
&\lesssim&\|F_1\|_{L^a((T,\infty),L^{\frac{rN}{N-s_c'}})}^{\lambda''}\|F_1\|_{L^a((T,\infty),L^\infty)}^{1-\lambda''}\\
&\lesssim&\|e^{i(\cdot-(T-\varepsilon^{-\beta}))\Delta^2}u(T-\varepsilon^{-\beta})-e^{i\cdot\Delta^2}u_0\|_{L^a((T,\infty),L^\frac{rN}{N+rs_c'})}^{\lambda''}\|F_1\|_{L^a((T,\infty),L^\infty)}^{1-\lambda''}\\
&\lesssim&\|F_1\|_{L^a((T,\infty),L^\infty)}^{1-\lambda''}.
\end{eqnarray*}
With the free Schr\"odinger operator decay,
 for $T\leq t$ and $p\geq\frac32+\frac\alpha N$, one gets via Proposition \ref{gag2},
\begin{eqnarray*}
\|F_1\|_{\infty}
&\lesssim&\int_0^{T-\varepsilon^{-\beta}}\frac1{(t-s)^{\frac N4}}\|(I_\alpha*|\cdot|^b|u|^p)|x|^b|u|^{p-1}\|_1\,ds\\
&\lesssim&\int_0^{T-\varepsilon^{-\beta}}\frac1{(t-s)^{\frac N4}}\|u(s)\|_{H^2}^{2p-1}\,ds\\
&\lesssim&(t-T+\varepsilon^{-\beta})^{1-\frac N4}.
\end{eqnarray*}
Thus, if $\frac N4>1+\frac1a$, it follows that
\begin{eqnarray*}
\|F_1\|_{L^a((T,\infty),L^r)}
&\lesssim&\|F_1\|_{L^a((T,\infty),L^\infty)}^{1-\lambda''}\\
&\lesssim&\Big(\int_T^\infty(t-T+\varepsilon^{-\beta})^{a[1-\frac N4]}\,dt\Big)^{\frac{1-\lambda''}a}\\
&\lesssim&\varepsilon^{(1-\lambda'')\beta[\frac N4-1-\frac1a]}.
\end{eqnarray*}
The condition $\frac N4>1+\frac1a$ is equivalent to $N\geq5$ and $(p-1)(2p-1)>\frac{4+2b+\alpha}{N-4}.$ Taking the polynomial function 
$$P_\alpha(X):=(X-1)(2X-1)-\frac{4+2b+\alpha}{N-4}:=(X-x_\alpha)(X-y_\alpha), \quad x_\alpha>0,$$
 the previous inequality is equivalent to $p>x_\alpha$. Thus, one concludes the proof by collecting the previous estimates.
\end{proof}
The proof of the scattering criterion follows like the case of an inhomogeneous local source term. 
\end{proof}
\subsection{Scattering}
Theorem \ref{sctr} about the scattering of energy global solutions to the focusing problem \eqref{S} follows with Proposition \ref{crt} via Lemma \ref{evac1}.


\end{document}